\documentclass{cimart}

\DeclareMathOperator\GL{GL}
\DeclareMathOperator\diag{diag}
\usepackage[symbol]{footmisc} 

\title[Elementary divisor rings with Dubrovin-Komarnytsky property]{Elementary divisor rings with \\ Dubrovin-Komarnytsky property\footnote{The current  article finishes the joint collaboration that had started prior to Professor Bohdan  Zabavsky's demise in August 2020.  
The research  was supported by the UAEU  UPAR  grants  G00002160 and G00003658.}}

\author{Victor Bovdi and Bohdan Zabavsky}

\authorinfo[V. Bovdi]{United Arab Emirates University, UAE}{vbovdi@gmail.com}

\abstract{%
    Elementary divisor rings were first introduced by Kaplansky in his seminal work. The purpose of this research is to extend Kaplansky’s study of commutative elementary divisor rings to certain classes of associative rings under weaker conditions than commutativity. We introduce two new classes of non-commutative rings: those with the $DK$-property (Dubrovin–Komarnytsky property) and those with the $D$-property (Dubrovin property), and investigate the structure of elementary divisor rings within these settings. Our main focus is on non-commutative rings of stable range $1$. For such rings, we develop a theory of reduction matrices, which allows us to construct and analyze new families of non-commutative elementary divisor rings.  In addition, we introduce the concept of an elementary element in a non-commutative ring. We prove that for Bézout domains of stable range $1$ with the $DK$-property, a ring $R$ is an elementary divisor ring if and only if every nonzero element of $R$ is elementary.
    }



\VOLUME{34}
\YEAR{2026}
\ISSUE{1}
\NUMBER{2}
\DOI{https://doi.org/10.46298/cm.16406}

\begin{document}

\vspace*{-1ex}
\begin{flushright}
    \textit{Dedicated to Professor M.~A.~Salim on the occasion of his 70th birthday.}
\end{flushright}

\tableofcontents


\section{Introduction}
Let $R$ be an associative (but not necessarily commutative) ring with $1\not=0$, let $U(R)$ be the group of units of $R$, and let $R^{n\times m}$ be the vector space  of ${n\times m}$ matrices over $R$ with $n,m\geq 1$. Let $\GL_n(R)$ be the group of units of the matrix ring $R^{n\times n}$.
The matrix $D:=\diag(d_{1},d_{2},\ldots,d_{s})\in R^{n\times m}$ means a (possibly rectangular) matrix  having $d_{1},\ldots, d_{s}$ (in which $s:=\min(n,m)$) on the main diagonal and zeros elsewhere.  By the main diagonal we mean the one beginning at the upper left corner. According to Kaplansky (see \cite[p.\,465]{7-kaplansky}),    a ring $R$ is called an {\it elementary divisor ring} if for any matrix $A$ over $R$ there exist invertible matrices $P$ and $Q$  of suitable sizes such that
\begin{equation}\label{FFF}
PAQ=D:=\diag(d_1,\ldots,d_k,0,\ldots,0),
\end{equation}
in which  $d_{i}$ is a total divisor of $d_{i+1}$, i.e., $Rd_{i+1}R\subseteq d_iR\cap Rd_i$ for each $i=1,\ldots, k-1$.

The class of elementary divisor rings is contained in the class of B\'ezout rings (for example, see \cite{7-kaplansky, Mon_Shchedryk, 1-zabavsky}), i.e.,  rings with nonzero unit in which every finitely generated one-sided ideal is a principal one-sided ideal. Note that  elementary divisor rings are Hermite rings. Hermite rings are    rings in which each  $1\times 2$ and $2\times 1$ matrix   has a  diagonal reduction, i.e., $(a,b)P=(c,0)$ and $Q(a, b)^T=(d,0)^T$, where    $a, b, c, d\in R$ and   $P, Q\in \GL_2(R)$  (see \cite{7-kaplansky, 1-zabavsky}). Elementary divisor rings have been studied by many authors (for example, see \cite{Ara_Goodearl_OMeara_Pardo, Bovdi_Zabavsky, Bovdi_Zabavsky_2, Bovdi_Shchedryk_2, Bovdi_Shchedryk, Chen_Abdolyousefi, Dopico_Noferini_Zaballa,  Dubrovin_82, 3-dubrovin, Gatalevich, Calugareanu_Pop_Vasiu, Menal_Moncasi, 2-zabkom}) and an overview of this topic  can be found in \cite{9-cohn, 7-kaplansky, Mon_Shchedryk, Shchedryk_2012, 1-zabavsky}.

It is well known \cite{Warfield} that any finitely presented module over a valuation ring is isomorphic to a direct sum of cyclic modules. Thus, a natural question arises: whether there are other classes of rings that satisfy the above property. In the articles of Kaplansky \cite {7-kaplansky} and Larsen, Lewis and Shores \cite{ Larsen_Lewis_Shores}, it was shown that the above problem and the problem of reducing an arbitrary matrix to a diagonal form \eqref{FFF} are equivalent over commutative rings. Thus, such a reduction   combines the theory of rings and the theory of modules.

In the article \cite{Menal_Moncasi}, the notion of stable range comes to the theory of diagonalizability of matrices from $K$-theory.  Using this notation,  one of the  classes of rings was characterized in \cite{Menal_Moncasi}, namely the class of regular rings, which is important for solving the problem of reducing an arbitrary matrix to a diagonal form \eqref{FFF}. During the last decade,  algebraic $K$-theory has been actively used for the  study  of elementary divisor rings. The use of  invariants as the stable range of rings is of  particular importance (for some examples, see \cite{Anderson_Juett_2012, Bovdi_Shchedryk_2, Bovdi_Shchedryk,  Juett_Williams_2017, 4-mcGovern,  Menal_Moncasi, Mortini_Rupp, Mon_Shchedryk, 4-zabavsky, 1-zabavsky,   Zabavsky_Gatalevych_19}). The aim of our study is an attempt to extend the results obtained for commutative rings to the case of non-commutative rings. Even the case of a  $2\times 2$  matrix ring over non-commutative rings shows that there are significant difficulties in realizing this project.  We have succeeded in obtaining some results with natural restrictions on non-commutative rings. In the general case, as part of his study of elementary divisor rings and related classes of rings, Kaplansky proved that if $R$ is an elementary divisor ring, then every finitely presented $R$-module is a direct sum of cyclic modules. The results of this article may therefore be of interest for the study of finitely presented modules over certain classes of noncommutative rings.

\section{Main Results and relations between them}

The {\it coboundary}  of a one-sided ideal $I$ of a ring $R$ is a  two-sided ideal which is equal to the intersection of all two-sided ideals which contain $I$.     Note that this definition is  left-right symmetric.

A nonzero element $a\in R$ is called a {\it right (left) duo} if $aR$ ($Ra$) is a two-sided ideal. Moreover,  if  $aR=Ra$,   then the element $a$ is called {\it duo}.

A ring   $R$ has the {\it  $D$-property (Dubrovin's property)}  \cite[p.\,33]{1-zabavsky} if for each  $a\in R$ there exists  $a_*\in R$ such that $RaR=a_*R=Ra_*$ (in other words, the  coboundary of $R$ is a   principal ideal).

In the sequel, we  consider  only  rings $R$ with the $D$-property. Examples of such rings are simple rings \cite[\S 4.2]{1-zabavsky}, quasi-duo elementary divisor rings \cite[Theorem 1]{2-zabkom}, and semi-local semi-prime elementary divisor rings \cite[Theorem 1]{3-dubrovin}. Several examples of such rings are given in \cite{Bovdi_Zabavsky_2}.

Two matrices $A$ and $B$ over a ring $R$ are called {\it equivalent} over the ring $R$ if there exist invertible matrices $P$ and $Q$ over $R$ of  suitable sizes such  that $A=PBQ$ and it will be  denoted by $A\sim B$.

Let $A=(a_{ij})\in R^{m\times n}$.  The coboundary of the right ideal generated by all  elements of the  matrix $A=(a_{ij})$  is  denoted by $A_*$,
that is
$A_*=\sum\limits_{i=1}^m\sum\limits_{j=1}^n Ra_{ij}R$.

A ring $R$  has  {\it  stable range $1$}  if the property $aR+bR=R$ implies $(a+bt)R=R$ for some $t\in R$.  Semi-local rings and  unit-regular rings \cite[p.\,45]{1-zabavsky} are examples of  rings of stable range $1$. Each  commutative   B\'ezout ring of  stable range $1$ is an elementary divisor ring. It is known \cite[Theorem 2]{4-zabavsky} that each   non-commutative  right B\'ezout ring  of stable range $1$ is a right Hermitian ring.

Our first result is of a technical nature, but it is actively used in what follows.

\begin{theorem}\label{ThR:1}
If  $R$ is  a B\'ezout ring  of stable range $1$, then
\begin{itemize}
\item[(i)] for any  $A\in  R^{2\times 2}$ there exist $z,\gamma,d\in R$ such that $A \sim \left[\begin{array}{cc}
z & \gamma \\
d & 0
\end{array}
\right]
$ and    $RzR=A_*$;

\item[(ii)] for any  $A\in R^{2\times 2}$ there exist $a,b,c\in R$ such that $A\sim
\left[\begin{array}{cc}
          a & 0 \\
          b & c
        \end{array}\right]$ and $RaR=A_*$.
\end{itemize}
\end{theorem}

A ring  $R$ has the {\it  $K$-property} (Komarnytsky's property)  (see \cite{Komarnitskii}) if  each  factor of a duo element is a duo  element.

Duo-rings and simple rings are rings with $K$-property. However,   a principal ideal domain $\mathbb{H}[x]$ over the classical quaternion divison ring $\mathbb{H}$ contains the  element $1+x^2$ which  is duo but not evert factor of $1+x^2=(1+ix)(1-ix)$ is  duo.

Of course, in any ring the central elements and invertible elements are always duo. More examples and relations between properties of such elements  can be found in \cite{Cheon_Kwak_Lee_Piao_Yun, 3-dubrovin, Gatalevich, Kim_Kwak_Lee, Kosan_Lee_Zhou, 2-zabkom, 1-zabavsky}.

We now introduce the following class of noncommutative rings. A ring is said to have the $DK$-property (Dubrovin–Komarnytsky property) if it satisfies both the $D$-property and the $K$-property.

\smallskip
Our second main result is the following.

\begin{theorem}\label{ThR:2} Let $R$ be an elementary divisor ring.  If $R$ has the $DK$-property, then  for each  matrix $A$ over $R$ there exists a matrix $C:=\diag(\varepsilon_1,\ldots, \varepsilon_k, 0,\ldots,0)$  such that $A\sim C$, $\varepsilon_1, \ldots, \varepsilon_{k-1}$ are duo elements  and \[ R\varepsilon_{i+1}R\subseteq R\varepsilon_i\cap \varepsilon_iR, \qquad  (i=1,\dots, k-1).
\]
\end{theorem}

We present one more class of non-commutative rings associated with the $DK$-property.

A ring $R$ has the {\it  elementary   $DK$-property} or the ($EDK$-property) if for each  matrix $A$ over $R$ there exist invertible matrices $P$ and $Q$ of suitable sizes such that
\[
PAQ=\diag(\varepsilon_1,\ldots, \varepsilon_k, 0,\ldots,0),
\]
in which  $R\varepsilon_{i+1}R\subseteq R\varepsilon_i\cap \varepsilon_iR$ for all $i=1,\dots, k-1$ and  $\varepsilon_{1},\ldots, \varepsilon_{k-1}$ are  duo elements of $R$.

Evidently, elementary divisor rings with the $DK$-property, simple elementary divisor rings and elementary divisor duo-rings are examples of  rings with the $EDK$-property.
\smallskip

Note that to date  we do not know examples of non-commutative rings that are $EDK$-rings but are not $DK$-rings. But for $PI$-rings holds the following.

\smallskip

\begin{theorem}\label{ThR:3}
A principal ideal domain  $R$ is an elementary divisor ring with the $DK$-property if and only if  $R$ is a  ring with the $EDK$-property.
\end{theorem}

Our next main result is the following.

\begin{theorem}\label{ThR:4}
A Hermite ring $R$ has the $EDK$-property if and only if  each $A\in R^{2\times 2}$ is equivalent to $\diag(\varepsilon,a)\in R^{2\times 2}$, in which $a\in R$ and
\[
\text{either}\quad   RaR\subseteq\varepsilon R=R\varepsilon \quad \text{or}\quad \varepsilon=0.
\]
\end{theorem}

The next  result extends a well-known  Kaplansky's  criterion \cite[Theorem 5.2,  p.\, 472]{7-kaplansky} for non-commutative Hermite rings.
\begin{theorem}\label{New_Cor:5} 
A Hermite ring $R$ with the $DK$-property  is an elementary divisor ring if and only if for each  $a,b,c\in R$ with the property
\begin{equation}\label{EQQ:3}
RaR+RbR+RcR=R
\end{equation}
there exist $p,q\in R$ such that $paR+(pb+qc)R=R$.
\end{theorem}

Let $a\in R$ such that $RaR=R$. There  exist $n\in \mathbb{N}$ and $u_1$, \dots, $u_n$, $v_1$, \dots, $v_n\in R$ such that
    \begin{equation}\label{eq-1}
     u_1av_1+u_2av_2+\dots+u_nav_n=1.
    \end{equation}
If $n\in \mathbb{N}$ is the minimal number satisfying  \eqref{eq-1}, then $a\in R$ is called {\it $n$-simple}. We mostly  concentrate  on  $2$-simple elements of the ring $R$. Recall that an element $a\in R$ is called {\it $2$-simple} if $u_1av_1+u_2av_2=1$ for some $u_1,u_2,v_1,v_2\in R$ and $n=2$ is minimal.

The structure of  elements of an  elementary divisor domain  with the  $DK$-property is presented by the following.

\begin{theorem}\label{ThR:5}
Let $R$ be an elementary divisor domain. If   $R$ has the $DK$-property, then  every   $a\in R\setminus\{0\}$ can be  written in the form $a=\alpha b=c\alpha$, where $\alpha\in U(R)$  and $b, c\in R$ are  $2$-simple elements.
\end{theorem}

    In the case of a ring of stable range $1$, we have the following.

\begin{theorem}\label{NewT:6} 
Let $R$ be a ring of stable range $1$. If $a\in R$ is a $2$-simple element, then
\[
\mathrm{diag}(a,a)\sim\mathrm{diag}(1,\Delta),\qquad\quad   (\Delta \in R).
\]
\end{theorem}

Please note that the question still remains open whether a non-commutative ring of stable range $1$ will be   an elementary divisor ring.

Now we use the following definition. A ring $R$ has the {\it $L$-property} if it follows  from the  condition $RaR=R$  that  $a\in U(R)$.

\begin{theorem}\label{ThR:6}
Let $R$ be a domain with the $D$-property. If  $R$ has an   $L$-property, then $R$ is a duo domain.
\end{theorem}

Note  that  a ring is called a {\it quasi-duo ring} if   every maximal one-sided ideal is a two-sided ideal. Every  $n$-simple element is invertible in a quasi-duo domain; and in a quasi-duo elementary divisor domain the $D$-property is always satisfied  \cite[Theorem~1]{2-zabkom}.

As a consequence of the previous theorem we have the following.

\begin{corollary}\label{theor-15} (see \cite[Theorems 1 and 2]{2-zabkom})
A quasi-duo B\'ezout domain $R$ is an elementary divisor domain if and only if $R$ is a duo domain.
\end{corollary}

Let $R$ be a B\'ezout domain. An element $a\in R\setminus {0}$ is called {\it finite} if each right and left ideal that contains $a$ is principal. In a B\'ezout domain this condition is equivalent to the a.c.c. for principal right and left ideals which contain the element $a\in R$.

\begin{corollary}\label{theor-17}
Let $R$ be a B\'ezout domain of stable range $1$ with the $DK$-property. If it follows  from the  condition $RaR=R$  that   $a\in R$ is a finite element, then  $R$ is an elementary divisor ring.
\end{corollary}

In a more general situation as in Theorem \ref{ThR:5}, we have the following.
\begin{theorem}\label{ThR:7}
Let $R$ be a  domain with the  $D$-property.  If it follows  from the  condition $RbR=R$  that   $b\in R$ is a finite element, then any element   $a\in R\setminus {0}$ can be written as
\[
a=\alpha f=\varphi \alpha,
\]
in which  $\alpha\in R$ is a duo element and $f, \varphi\in R$ are  finite elements.
\end{theorem}


Let $R$ be a B\'ezout domain  with the $DK$-property.    An element   $a\in R$ is called {\it elementary} if $RaR=R$  and  for each  $b,c\in R$ there exist  $p,q\in R$ such that
\[
paR+(pb+qc)R=R.
\]
Note that invertible and finite elements are  elementary  elements  by \cite[Theorem~2]{6-beauregard}.


Our last result is the following.

\begin{theorem}\label{ThR:8}
Let $R$ be a B\'ezout domain of stable range $1$ with the $DK$-property.   The ring  $R$ is an elementary divisor ring  if and only if each nonzero element of $R$ is elementary.
\end{theorem}

As a consequence we have the following.

\begin{corollary}\label{Cor:8}
Each  quasi-duo elementary divisor domain of  stable range $1$ is a duo domain.
\end{corollary}

\section{Proofs}

Let $A=(a_{ij})\in R^{m\times n}$.  The coboundary of the right ideal generated by all  elements of the  matrix $A=(a_{ij})$  is  denoted by $A_*$,
that is $A_*=\sum\limits_{i=1}^m\sum\limits_{j=1}^n Ra_{ij}R$.

We start with the following well-known result.

\begin{lemma}\label{LLL:1}
    If $A$ and $B$ are equivalent matrices over $R$ then $A_*=B_*$.
\end{lemma}
\begin{proof}
If $A=(a_{ij})=PBQ$, where  $B=(b_{ij})$, then
\[
a_{ij}\in \sum\limits_{k}\sum\limits_{s} Rb_{ks}R\quad \text{ and }\quad b_{ij}\in \sum\limits_{i}\sum\limits_{j} Ra_{ij}R.
\]
It follows that
$\sum\limits_{i}\sum\limits_{j} Ra_{ij}R= \sum\limits_{k}\sum\limits_{s} Rb_{ks}R$, so
 $A_*=B_*$.
\end{proof}

We shall freely use  the following well-known results.

\begin{lemma} \label{LLL:2}
The following statements hold:
\begin{itemize}

\item[(i)] each B\'ezout ring of stable range $1$ is Hermite  (see \cite[Theorem 2]{4-zabavsky});

\item[(ii)] if  $R$ is  a right B\'ezout ring of stable range $1$, then for any $a,b\in R$ there exist $x,d\in R$ such that $a+bx=d$ and $aR+bR=dR$ ($a+xb=d$ and $Ra+Rb=Rd$, respectively) (see \cite[Proposition 6]{5-zabavsky});

\item[(iii)] if $a\in R$ is right duo,  then every factor of $a$ is a left factor. Moreover,  if $a$ is duo,  then every proper factor of $a\in R$ is a proper left factor (see \cite[Proposition 2]{6-beauregard}).
\end{itemize}
\end{lemma}

\begin{remark}\label{rem-1}
Let $R$ be a left B\'ezout ring  of stable range $1$. For each  $a,b\in R$ there  exist $x,y, d\in R$ such that $xa+yb=d$, so $Ra+Rb=Rd$.
\end{remark}

\begin{proof}[Proof of Theorem \ref{ThR:1}]
(i) Since $R$ is a Hermite ring (see Lemma~\ref{LLL:2}(i)),  up to the equivalence of matrices, we can assume   that $A=\left[\begin{array}{cc}
          \alpha & 0 \\
          \beta & \gamma
\end{array}
\right]$, so  $x\alpha+\beta=d$ and  $R\alpha+R\beta=Rd$ for some $x,d \in R$ by  Lemma~\ref{LLL:2}(ii) and Remark~\ref{rem-1}. Thus
\[
        \left[\begin{array}{cc}
          x & 1 \\
          1 & 0
        \end{array}\right]\left[\begin{array}{cc}
          \alpha & 0 \\
          \beta & \gamma
        \end{array}\right]=\left[\begin{array}{cc}
          x\alpha+\beta & \gamma \\
          \alpha & 0
        \end{array}\right]
\]
in which  $\alpha=\alpha_0d$ for some $\alpha_0\in R$.

Let $\gamma R+d R=zR$ and $\gamma y+d=z$ for some $y\in R$. Evidently,
\[
        \left[\begin{array}{cc}
          d & \gamma \\
          \alpha & 0
        \end{array}\right]\left[\begin{array}{cc}
          1 & 0 \\
          y & 1
        \end{array}\right]=\left[\begin{array}{cc}
          d+\gamma y & \gamma \\
          \alpha & 0
        \end{array}\right]=\left[\begin{array}{cc}
          z & \gamma \\
          \alpha & 0
        \end{array}\right],
\]
where $d=zt$ and  $\gamma=z\gamma_0$ for some $t,\gamma_0\in R$. It means that  $A\sim  \left[\begin{array}{cc}
          z & \gamma \\
          d & 0
\end{array}
\right],$
where $d=zt$ and $\gamma=z\gamma_0$. This yields  that $Rd R\subseteq RzR$ and $R\gamma R\subseteq RzR$, i.e.,
\[
RzR+R\gamma R+Rd R=RzR.
\]
Finally,   $A_*=RzR+R\gamma R+Rd R=RzR$ by Lemma~\ref{LLL:1}.

(ii) For each matrix $A\in R^{2\times 2}$ there exists an equivalent matrix
$B=\left[\begin{array}{cc}
 z & \gamma \\
\delta & 0
\end{array}\right]$
such that $RzR=A_*$ by Theorem \ref{ThR:1}. If  $t\in R$ such that $zR+\gamma R=(z+\gamma t)R$ (see Lemma~\ref{LLL:2}(ii)), then
\[
        \left[\begin{array}{cc}
          z & \gamma \\
          \delta & 0
        \end{array}\right]\left[\begin{array}{cc}
          1 & 0 \\
          t & 1
        \end{array}\right]=\left[\begin{array}{cc}
          z+\gamma t & \gamma \\
          \delta & 0
        \end{array}\right].
\]
Let $z+\gamma t=a$ and $\gamma=as$ for some $s\in R$. Obviously,
\[
        \left[\begin{array}{cc}
          z+\gamma t & \gamma \\
          \delta & 0
        \end{array}\right]=\left[\begin{array}{cc}
          a & as \\
          \delta & 0
        \end{array}\right]
\]
        and
\[
        \left[\begin{array}{cc}
          a & as \\
          \delta & 0
        \end{array}\right]\left[\begin{array}{cc}
          1 & -s \\
          0 & 1
        \end{array}\right]=\left[\begin{array}{cc}
          a & 0 \\
          \delta & -\delta s
        \end{array}\right]=\left[\begin{array}{cc}
          a & 0 \\
          b & c
        \end{array}\right].
\]
Hence $RzR=RaR$  because $\left[\begin{array}{cc}
          1 & 0 \\
          t & 1
       \end{array}\right]\in GL_2(R)$. Also  $RbR\subseteq RaR$, and $RcR\subseteq RaR$. \end{proof}

\begin{lemma}\label{LLL:3}
Let $R$ be a right Hermite ring. For each $a,b\in R$ there exist $a_0,b_0, d\in R$ such that $a=da_0$, $b=db_0$ and  $a_0R+b_0R=R$.
\end{lemma}
\begin{proof}
Since $R$ is  Hermite, $(a,b)P=(d,0)\in R^{2\times 2}$ for an invertible
$P=\left[\begin{array}{cc}
u & * \\
v & *
\end{array}\right]$. If
$P^{-1}:=\left[\begin{array}{cc}
          a_0 & b_0 \\
          * & *
\end{array}\right]$,  then
\[
\textstyle
(d,0)P^{-1}=(d,0)\left[\begin{array}{cc}
          a_0 & b_0 \\
          * & *
        \end{array}\right]=(a,b),
\]
so   $da_0=a$, $db_0=b$ and $a_0u+b_0v=1$.
\end{proof}

\begin{lemma}\label{LLL:4}
Each  Hermite ring  $R$ with the  $D$-property  is an elementary divisor ring if and only if every matrix $A=(a_{ij})$ over $R$ with the property   $\sum\limits_{i}\sum\limits_{j} Ra_{ij}R=R$ has  a form  \eqref{FFF}.
\end{lemma}
\begin{proof} Since the proof of the ``if'' part  is obvious, we start with the proof of the ``only if'' part. Let $Ra_{ij}R= a_{ij}^*R= Ra_{ij}^*R$ for  each $i$, $j$ and for some duo-element $\alpha\in R$ we have
\[
\sum\limits_{i}\sum\limits_{j} Ra_{ij}R=\sum\limits_{i}\sum\limits_{j} a_{ij}^*R=\sum\limits_{i}\sum\limits_{j} Ra_{ij}^*=\alpha R=R\alpha.
\]
That yields  $a_{ij}=\alpha a_{ij}^0$ and $ A=\mathrm{diag}(\alpha,\dots,\alpha)A_0$, where $A_0= (a_{ij}^0)$.
Since $R$ is a Hermite ring,  $\sum\limits_{i}\sum\limits_{j} Ra_{ij}^0R=R$ by Lemma \ref{LLL:3}.
\end{proof}

\begin{proof}[Proof of Theorem \ref{ThR:2}]
Since  $R$ is an elementary divisor ring, we  assume   that  
\[A=\diag(d_1,\ldots, d_k, 0,\ldots, 0)\] in which $Rd_{i+1}R\subseteq Rd_i\cap d_iR$ for  $i=1,\dots, k-1$. The ring  $R$ has the $D$-property, so
\[
Rd_{i+1}R=d_{i+1}^*R=R d_{i+1}^*,
\]
i.e., $d_{i+1}^*$ is a  duo  element. From $d_{i+1}^*R=Rd_{i+1}^*\subseteq d_iR\cap Rd_i$ and the definition of the  $K$-property  we obtain  that each $d_i$ is  a duo element.
\end{proof}

\begin{proof}[Proof of Theorem \ref{ThR:3}]
Since a principal ideal domain is an elementary divisor ring with the $D$-property \cite[p.\,5]{1-zabavsky}, the proof of the ``if'' part follows from  Theorem~\ref{ThR:2}.

Let $a,z\in R$ such that $a$ is duo and $z$ is a divisor of $a$. Since $a$ is duo,  $z$ is a left and right divisor of $a$  by Lemma \ref{LLL:2}(iii).  Consequently,   $RaR\subseteq zR\cap Rz$ and
$A=\left[\begin{array}{cc} z & 0 \\  0 & a \end{array}\right]$  has  a normal canonical form \eqref{FFF}. Since $R$ is an elementary divisor ring with the $DK$-property,   $A\sim \left[\begin{array}{cc} x & 0 \\  0 & b \end{array}\right]$ in which  $RbR\subset xR= Rx$.  Moreover,  $R/zR\cong R/xR$ as a right $R$-module (see \cite[Theorem 2.8]{Amitsur}) and  $x$ is a duo element. Consequently,   $z$ is a duo element by \cite[Exercise 4.2.16]{9-cohn}.
\end{proof}

\begin{proof}[Proof of Theorem \ref{ThR:4}]
Taking into  account  that the elements $\varepsilon_1, \ldots, \varepsilon_{k-1}$ are  duo in the definition of the $EDK$-property, the proof is the same as  the proof of \cite[Theorem~51]{7-kaplansky}.
\end{proof}

\begin{proof}[Proof of Theorem \ref{New_Cor:5}]
Let \eqref{EQQ:3} hold   for some $a,b,c\in R$. For the matrix $A=\left[\begin{array}{cc}
                                                                         a& b \\
                                                                         0 & c
                                                                       \end{array}\right]$
there exist invertible matrices $P:=\left[\begin{array}{cc}p&q\\ *&*\end{array}\right]$ and  $Q:=\left[\begin{array}{cc}  x& * \\ y & * \end{array}\right]$ such that
\[
PAQ=\diag(z, d)\qquad \text{and}\qquad  RdR\subseteq zR=Rz.
\]
Evidently,  $RzR+RdR=RzR=RaR+RbR+RcR=R$. Since $zR=Rz$,  then
\[
RzR=zR=Rz \quad\text{and}\quad  zR=Rz=R,
\]
i.e., $z\in U(R)$. On the other hand,  we have that $paxz^{-1}+(pb+qc)yz^{-1}=1$ which we intended  to prove in the first place.

The ring $R$ is Hermite, so according  to  Lemma \ref{LLL:4} and Theorem~\ref{ThR:4} for the proof of the ``only if'' part  it is enough to prove for the matrices $A=\left[\begin{array}{cc}
a& b \\
0 & c
\end{array}\right]$,  where $a,b,c\in R$ which satisfies \eqref{EQQ:3}.

According to the condition of our theorem, there exist $p,q\in R$ such that \[paR+(pb+qc)R=R.\] Let $pax+(pb+qc)y=1$.  Since $pR+qR=R$ and   $Rx+Ry=R$,  then  there exist matrices
$P:=\left[\begin{array}{cc}p&q\\ *&*\end{array}\right]\in\GL_2(R)$ and $Q:=\left[\begin{array}{cc}  x& * \\ y & * \end{array}\right]\in \GL_2(R)$ such that

\[
PAQ=\left[\begin{array}{cc}1&*\\ *&*\end{array}\right]\sim\left[\begin{array}{cc}1&0\\ 0&\Delta\end{array}\right],\qquad \qquad(\Delta\in R).
\]
\end{proof}

\begin{proof}[Proof of Theorem \ref{ThR:5}]
For every $a\in R\setminus\{0\}$, we have $RaR=\alpha R=R\alpha$ because $R$ is a ring  with the $DK$-property. This  yields $a=\alpha b=c\alpha$ for some $b,c\in R$.

Since $R$ is a domain, $RbR=R$ and $RcR=R$. Using the fact that  $R$ is an elementary divisor domain with the $DK$-property, we have
\begin{equation}\label{EQQ:4}
\left[\begin{array}{cc}
      a & 0 \\
      0 & a
\end{array}\right]P=Q\left[\begin{array}{cc}
      z & 0 \\
      0 & d
\end{array}\right]
\end{equation}
in which  $P=(p_{ij})\in \GL_2(R)$, $Q=(q_{ij})\in \GL_2(R)$ and $RdR\subseteq zR=Rz$.

Let us show that $d\ne 0$. Indeed,  if $d=0$, then  $ap_{12}=0$ and  $ap_{22}=0$  by \eqref{EQQ:4}. Since $a\ne 0$ and $R$ is a domain,  $p_{12}=p_{22}=0$ which is impossible because
$P=\left[\begin{array}{cc}
      p_{11} & p_{12} \\
      p_{21} & p_{22}
\end{array}\right]$
is invertible. Hence  $d\ne 0$ and   $z$ is a  duo element. Moreover,  $z\in U(R)$, because  $R=RaR=RzR$. Without loss of generality,  we can assume that  $z=1$. From \eqref{EQQ:4}, we obtain that
    \begin{equation}\label{EQQ:5}
    ap_{11}=q_{11}\qquad\text{and}\qquad  ap_{12}=q_{21}.
    \end{equation}
Taking into account that  $Q\in \GL_2(R)$,  we deduce  that    $uq_{11}+vq_{12}=1$ for some $u,v\in R$, so $uap_{11}+vap_{12}=1$  by \eqref{EQQ:5}. Consequently,   $a$ is a $2$-simple element.
\end{proof}

\begin{proof}[Proof of Theorem \ref{NewT:6}]
 Since the element $a$ is $2$-simple, there exist $u_1, u_2, v_1, v_2 \in R$ such that $u_1 a v_1 + u_2 a v_2 = 1$. Hence
$u_1aR+u_2aR=R$, because  $R$ is a ring of  stable range $1$. It follows that
\[
u_1a+u_2at=w_1\in U(R) \quad (\text{for some} \, t\in R)
\]
and  $Ra+Rt=R$. Obviously,  $xa+t=w_2\in U(R)$ for some $x\in R$. Since  $t=w_2-xa$ and  $u_1a+u_2at=w_1$, we have that
\[
\begin{split}
u_1a+u_2 a(w_2-xa)&=u_1a+u_2aw_2-u_2axa\\&=(u_1-u_2ax)a+u_2aw_2=w_1.
\end{split}
\]
This yields  $Ra+Raw_2=R$ and $sa+aw_2=w_3\in U(R)$ for some $s\in R$. Consequently,
\[
\begin{split}
    \left[\begin{array}{cc}
      s & 1 \\
      1 & 0
    \end{array}\right]\left[\begin{array}{cc}
      a & 0 \\
      0 & a
    \end{array}\right]&\left[\begin{array}{cc}
      1 & 0 \\
      0 & w_2
    \end{array}\right]\left[\begin{array}{cc}
      1 & 0 \\
      1 & 1
    \end{array}\right]=\\&=\left[\begin{array}{cc}
      sa+aw_2 & aw_2 \\
      a & 0
    \end{array}\right]
    =\left[\begin{array}{cc}
      w_3 & aw_2\\
      a & 0
    \end{array}\right]
\end{split}
\]
and   $\left[\begin{array}{cc}
      w_3 & aw_2\\
      a & 0
\end{array}\right]\sim\mathrm{diag}(1,\Delta)$ for some $\Delta\in R$.
\end{proof}

Now  we consider  rings of  stable range $1$ with the $D$-property. According to Lemma~\ref{LLL:4},  all  possible diagonal reductions to the form \eqref{FFF} for rings $R$ are conditional  on elements $a\in R$ such that $RaR=R$.

\begin{lemma}\label{LLL:7}
Let $R$ be a domain  with the $D$-property. Each  $a\in R\setminus\{0\}$ can be written in the form
\[
a=\alpha b=c\alpha,
\]
where $\alpha$ is a duo element and $b,c$  are $n$-simple elements for some  $n\in \mathbb{N}$.
\end{lemma}
\begin{proof}
Let $a\in  R\setminus \{0\}$ and $RaR=\alpha R=R\alpha$. It follows that  $a=\alpha b=c\alpha$ for some $b,c\in R$, so $\sum_{i=1}^nu_i\alpha  b v_i=\alpha$ because  $\sum_{i=1}^nu_iav_i=\alpha$ for some $u_1,\dots, u_n, v_1,\dots, v_n\in R$. Since $\alpha$ is a duo element,  $\alpha\sum_{i=1}^nu_i' b v_i=\alpha$, i.e., $\sum_{i=1}^nu_i' b v_i=1$ for some $u_1', \dots, u_n'\in R$. Thus  $RbR=R$. The proof of $RcR=R$  is similar.
\end{proof}

\begin{proof}[Proof of Theorem \ref{ThR:6}]
Follows from Lemma \ref{LLL:7}.
\end{proof}

\begin{proof}[Proof of Corollary \ref{theor-15}]
Each quasi-duo ring is a ring in which every maximal   one-sided ideal is a two-sided ideal. Since any $n$-simple element is invertible, then the $D$-property is always satisfied  \cite[Theorem~1]{2-zabkom} in a quasi-duo domain and in a quasi-duo elementary divisor domain.
\end{proof}

\begin{proof}[Proof of Corollary \ref{theor-17}]
The ring $R$ is  Hermite by \cite[Theorem 2]{4-zabavsky}.  According to Theorems~\ref{ThR:1} and  \ref{ThR:4},  it is sufficient to show our statement  for  a  matrix $A$ of the form
$\left[\begin{array}{cc}
a& 0 \\
b & c
\end{array}\right]$ in which   $RaR=R$.  Since   $a$ is  a finite element, it is   evident that   $A\sim \left[\begin{array}{cc}
                                                   f& 0 \\
                                                   0 & d
                                                 \end{array}\right]$
in which $RdR\subseteq f R\cap Rf$ by \cite[Theorem 6]{10-zabavsky}. Consequently,  $R$ is an elementary divisor ring.
\end{proof}

\begin{proof}[Proof of Theorem \ref{ThR:7}]
Follows from  Lemma \ref{LLL:7}.
\end{proof}

\begin{proof}[Proof of Theorem \ref{ThR:8}]
Let $a,b,c\in R$ such that $a\ne 0$ and  $RaR=R$. Since $R$ is an elementary divisor ring with the $DK$-property, for the matrix $A=\left[\begin{array}{cc}
                         a& b \\
                         0 & c
\end{array}\right]$ there exist invertible matrices
$P:=\left[\begin{array}{cc}p&q\\ *&*\end{array}\right]$ and $Q:=\left[\begin{array}{cc}  u& * \\ v & * \end{array}\right]$
such that $PAQ=\left[\begin{array}{cc}z&0\\ 0&d\end{array}\right]$     where $d=0$ or $RdR\subseteq zR=Rz$ for some  $z,d\in R$. Since $a\ne 0$ and $R$ is a domain, the case  $d= 0$ is impossible.

Evidently,  $RdR\subseteq RzR$ and
\[
RzR+RdR=RaR+RbR+RcR=R,
\]
so $RzR=R$. Since $z$ is a duo element, we have    $z\in U(R)$.  Clearly,  $z=pau+(pb+qc)v$  and
\[
pauz^{-1}+(pb+qc)vz^{-1}=1,
\]
 i.e., $paR+(pb+qc)R=R$.

{\it Proof of the ``only if'' part.}   Since each B\'ezout domain is a Hermite ring,   it is sufficient to prove our statement for the matrices of the form   $A=\left[\begin{array}{cc}
                         a& b \\
                         0 & c
\end{array}\right]$
in which $a\ne 0$ and $RaR=R$ by Lemma~\ref{LLL:7} and Theorem~\ref{ThR:4}. According to assumptions, there exist $p,q\in R$ such that
\[
paR+(pb+qc)R=R,
\]
that is, $pau+(pb+qc)v=1$, where  $u,v\in R$. Since $R$ is Hermite,  there exist matrices   $P:=\left[\begin{array}{cc}p&q\\ *&*\end{array}\right]\in \GL_2(R)$ and  $Q:=\left[\begin{array}{cc}  u& * \\ v & * \end{array}\right]\in \GL_2(R)$ such that
\[
m    PAQ=\left[\begin{array}{cc}1&*\\ *&*\end{array}\right]\sim \left[\begin{array}{cc}
      1& 0 \\
      0 & t
\end{array}\right],
\qquad \qquad (t\in R).
\qedhere
\]
\end{proof}

\EditInfo{August 26, 2025}{September 29, 2025}{Ivan Kaygorodov}


\end{document}